\documentclass[11pt]{article}
\usepackage{amsgen,amsmath,amstext,amsbsy,amsopn,amsfonts,amssymb}
\newtheorem{thm}{Theorem}
\newtheorem{lem}[thm]{Lemma}
\newtheorem{cor}[thm]{Corollary}
\newtheorem{prp}[thm]{Proposition}
 \newtheorem{df}[thm]{Definition}
 \newtheorem{q}[thm]{Question}
 \newtheorem{rem}[thm]{Remark}
 \newtheorem{example}[thm]{Example}
 \newtheorem{examples}[thm]{Examples}
\newtheorem{clm} [thm] {Claim}

\newenvironment{proof-thm3}{\paragraph{Proof of Theorem 3}}{}
\newenvironment{proof-prp12}{\paragraph{Proof of Proposition 12}}{}

\newenvironment{ack}{\paragraph{Acknowledgements :}}{}
 
 \newcommand{\esp}{\textrm{  }}
 \newcommand{\n}{\vert\vert}

\def\bsq{\blacksquare\medskip}
\def\n{\noindent}
\def\H{\mathcal{H}}

\def\R{\mathbb{R}}

\def\lp{\ell_{p}}
\def\l2{\ell_{2}}
\def\Lp{L_{p}([0,1])}
\def\Lp(X){L_{p}(X,\mu)}
\def\pc{p_{c}}

\def\T{(T)}
\def\Tlp{(T_{\ell_{p}})}
\def\Flp{(F_{\ell_{p}})}

\def\H1{H^{1}}

\def\n{\vert\vert}

\begin{document}

\title{Fixed-point spectrum for group actions by affine isometries on $L_{p}$-spaces.
\footnote{This work will appear at Annales de l'Institut Fourier.}}
\author{Omer Lavy and Baptiste Olivier \footnote{This work was done while both authors worked at the TECHNION. The second author was supported by ERC Grant no. 306706.}}
\date{}
\maketitle
\begin{abstract}
The fixed-point spectrum of a locally compact second countable group $G$ on $\ell_{p}$ is defined to be the set of $p\geq1$ such that every action by affine isometries of $G$ on $\ell_{p}$ admits a fixed-point. We show that this set is either empty, or is equal to a set of one of the following forms : $[1,\pc[$, $[1,\pc[\backslash\{2\}$  for some $1\leq\pc\leq\infty$, or $[1,\pc]$, $[1,\pc]\backslash\{2\}$ for some $1\leq\pc<\infty$. This result is closely related to a conjecture of C. Drutu which asserts that the fixed-point spectrum is connected for isometric actions on $L_{p}(0,1)$.\\
More generally, we study the topological properties of the fixed-point spectrum on $L_{p}(X,\mu)$ for general measure spaces $(X,\mu)$, and show partial results toward the conjecture for actions on $L_{p}(0,1)$.
In particular, we prove that the spectrum associated with actions with linear part $\pi$ is either empty, or an interval of the form $[1,p_{c}]$ ($p_{c}\geq1$) or $[1,\infty[$, whenever $\pi$ is an orthogonal representation associated to a measure-preserving ergodic action on a finite measure space $(X,\mu)$. 
\end{abstract}


\section{Introduction}

Group actions on Banach spaces is a large topic related to many areas of mathematics : group cohomology, Kazhdan's property $(T)$, and fixed-point properties. In \cite{BFGM}, Bader, Furman, Gelander and Monod studied group actions by isometries on Banach spaces. The authors of \cite{BFGM} provide several results concerning property $(F_{L_{p}(0,1)})$, the fixed-point property for group actions by affine isometries on the space $L_{p}([0,1],\lambda)$ (abbreviated $L_{p}(0,1)$), where $p\geq1$ and $\lambda$ denotes the Lebesgue measure. In particular, they show that a locally compact second countable group with the fixed-point property $(F_{L_{p}(0,1)})$ has the Kazhdan's property $(T)$ when $p>1$, and that these properties are equivalent when $1< p\leq2$ (see Theorem A and Theorem 1.3 in \cite{BFGM}). Later the main result of \cite{BGM} was that these properties are equivalent when $1\leq p\leq2$ \\

Some groups have property $(F_{L_{p}(X,\mu)})$ for all $p\geq1$ and all standard measure space $(X,\mu)$: a more general result states that higher rank groups have property $(F_{L_{p}(\mathcal{M})})$ for all $p\geq1$ and all von Neumann algebra $\mathcal{M}$ (see Theorem B in \cite{BFGM}, Theorem 1.6 in \cite{O1}, and \cite{LS} for a stronger result). See also \cite{M1} and \cite{M2} where analogous 
results are established for universal lattices $SL_{n}(\mathbb{Z}[x_{1},...,x_{k}])$ ($n\geq4$).\\

On the other hand, there exist Kazhdan groups which do not have property $(F_{L_{p}(0,1)})$ for some $p>2$. For instance, hyperbolic groups (and among them co-compact lattices in $Sp(n,1)$ have property $(T)$) admit a metrically proper action by affine isometries on $\ell_{p}$, as well as on $L_{p}(0,1)$, for $p$ large enough (see \cite{Yu} and \cite{Ni}).\\

The main motivation of this article is the following conjecture of C. Drutu: for every topological group $G$, there exists $\pc\geq 1$ ($p_c$ can be also infinite), such that $G$ has property $(F_{L_{p}(0,1)})$ for $1\leq p<\pc$, and $G$ does not have property $(F_{L_{p}(0,1)})$ for $p>\pc$. further more, if $p_c>1$, then $p_c >2$ (see Question 1.8 in the introduction of \cite{CDH}). It seems that this question has its root in an older question raised by M. Gromov (see \cite{G} D.6 p158). We introduce the following set and we use a similar terminology as in \cite{No}.
\begin{df}\label{df-spectrum}{\rm
Let $(X,\mu)$ be a standard Borel measure space, and let $G$ be a topological group. The set
$$\mathcal{F}_{L_{\infty}(X,\mu)}(G)=\{\esp p\geq 1\esp \vert \esp G\textrm{ has property }(F_{L_{p}(X,\mu)})\esp\}$$
is called the fixed-point spectrum of $G$ (for affine isometric actions) on $L_{p}(X,\mu)$-spaces.
}
\end{df}
We use the notation $\mathcal{F}_{L_{\infty}(X,\mu)}(G)$ since it makes sense in a more general context, that is for actions on non-commutative $L_{p}$-spaces; in that case $L_{\infty}(X,\mu)$ is replaced by a general von Neumann algebra $\mathcal{M}$, and the fixed-point spectrum is denoted by $\mathcal{F}_{\mathcal{M}}(G)$.\\

The conjecture of C. Drutu can be rephrased as follows : the set $\mathcal{F}_{L_{\infty}(0,1)}(G)$ is connected for all group $G$. Our first main result is an answer to the same question, when replacing $([0,1],\lambda)$ by a discrete measure space. We denote by $\ell_{\infty}$ the space of all bounded infinite sequences of complex numbers, and $\ell_{p}$ the space of $p$-summable sequences in $\ell_{\infty}$, equipped with the $p$ norm. 

\begin{thm}\label{thm-main-lp}
Let $G$ be a locally compact second countable group. Then one of the following equalities holds :\\
- $\mathcal{F}_{\ell_{\infty}}(G)=\emptyset$;\\
- $\mathcal{F}_{\ell_{\infty}}(G)=[1,\pc[$ for some $1< \pc\leq\infty$;\\
- $\mathcal{F}_{\ell_{\infty}}(G)=[1,\pc]$ for some $1\leq\pc<\infty$;\\
- $\mathcal{F}_{\ell_{\infty}}(G)=[1,\pc[\backslash\{2\}$ and $2<\pc\leq\infty$;\\
- $\mathcal{F}_{\ell_{\infty}}(G)=[1,\pc]\backslash\{2\}$ and $2<\pc<\infty$.
\end{thm}

Hence the spectrum $\mathcal{F}_{\ell_{\infty}}(G)$ can be the empty set, an interval, or the union of two disjoint intervals : we will show that these three situations can occur, depending on the group $G$ considered.\\

In the special case of countable groups, a similar result as Theorem \ref{thm-main-lp} was also proved independently of our work in \cite{C}, for the case of discrete countable groups and real $\ell_{p}$-spaces. Although some ideas in our proof are similar to the ones in \cite{C} (see section 4.1 for further precisions), our overall approach is different: Theorem \ref{thm-main-lp} is obtained by studying the topological properties of the fixed-point spectrum, while the approach in \cite{C} proves relationships with expander graphs theory. As shown by our Theorem \ref{thm-main-mp} and results from section 5 discussed below, the approach of the current paper permits to go beyond $\ell_{p}$-spaces, and generalizes to other $L_{p}$-spaces.\\


Let us now discuss the fixed-point spectrum of actions on general $L_{p}(X,\mu)$-spaces. We refer to section 2 for basic facts and definitions related to isometric group representations on $L_{p}$-spaces. Let $1\leq p<\infty$, $p\neq2$. Let $\pi^{p}:G\rightarrow O(L_{p}(X,\mu))$ be an orthogonal representation. There is a natural family $(\pi^{q})_{q\geq1}$ of orthogonal representations $\pi^{q}:G\rightarrow O(L_{q}(X,\mu))$ associated to $\pi^{p}$, namely the conjugate representations of $\pi^{p}$ by the Mazur maps $M_{p,q}$. Referring to a theorem by Banach and Lamperti recalled in Section 2, we will call \textit{(BL) representations} the representations of the form $\pi^{q}$, $1\leq q<\infty$. In the sequel, we will distinguish (BL) orthogonal representations from other possible classes of orthogonal representations in the case $q=2$. 
Then we can define the fixed-point spectrum of affine isometric actions with linear parts $(\pi^{q})_{q\geq1}$ as 
$$\mathcal{F}_{L_{\infty}(X,\mu)}(G,(\pi^{q})_{q\geq1})=\{\esp q\geq1\esp\vert\esp \H1(G,\pi^{q})=\{0\} \esp\}$$ 
where $\H1(G,\pi^{q})$ is the first cohomology group of $G$ with coefficients in $\pi^{q}$ (see section 2.3 for a definition). For simplicity in notations, we will denote $\mathcal{F}_{L_{\infty}(X,\mu)}(G,(\pi^{q})_{q\geq1})$ by $\mathcal{F}_{L_{\infty}(X,\mu)}(G,\pi^{p})$ for $p\neq2$ (but this set does not depend on $p$). In the sequel, we will say that $\pi^{p}$ is measure-preserving (resp. ergodic) if the associated action on $(X,\mu)$ is measure-preserving (resp. ergodic). We say that $\pi^{p}$ is positive if $\pi^{p}(g)f\geq 0$ for all $g\in G$ and all $f\geq 0$, $f\in L_{p}(X,\mu)$.\\

The following theorem describes the form of the spectra $\mathcal{F}_{L_{\infty}(X,\mu)}(G,\pi)$ relative to measure-preserving ergodic actions.

\begin{thm}\label{thm-main-mp}
Let $G$ be a second countable locally compact group. Let $(X,\mu)$ be a finite measure space, and $\pi:G\rightarrow O(L_{p}(X,\mu))$ be a (BL) measure-preserving ergodic representation. Then we have:\\
- The spectrum $\mathcal{F}_{L_{\infty}(X,\mu)}(G,\pi)$ is an interval or is empty.\\
- If $G$ has property $(T)$, the spectrum 
$\mathcal{F}_{L_{\infty}(X,\mu)}(G,\pi)$ is an interval of the form $[1,p_{c}]$ or $[1,\infty[$ for some $p_{c}>2$.
\end{thm}

Theorem \ref{thm-main-mp} is in contrast with the following well-known fact (detailed in section 5.2): if $G$ does not have property $(T)$, there exists some (BL) measure-preserving ergodic representation $\rho:G\rightarrow O(L_{p}(X,\mu))$
such that $\mathcal{F}_{L_{\infty}(X,\mu)}(G,\rho)$ is empty.\\

Theorems \ref{thm-main-lp} and \ref{thm-main-mp} are in the spirit of C. Drutu's conjecture. Other partial results in that direction are proved in the paper. We study topological properties of the fixed-point spectrum. We discuss general arguments for closedness (resp. openness) of fixed-point spectra in section 3 (resp. section 5.1). Moreover, we use recent results from \cite{BN} about deformation of cohomology to show the connectedness of spectra $\mathcal{F}_{L_{\infty}(0,1)}(G,\pi)$ under some additional assumptions on $G$ and $\pi$.\\


The paper is organized as follows. In Section 2, we recall general facts and results we will need about isometric group actions on $L_{p}$-spaces. In Section 3, we show general results concerning the closedness of the fixed-point spectrum. Section 4 is devoted to the proof of Theorem \ref{thm-main-lp}. In Section 5, we prove various partial results toward a proof of the conjecture of C. Drutu concerning the case of $L_{p}(0,1)$.


\begin{ack}
Authors are grateful to V. Lafforgue for many comments that improve the quality of the paper. They also would like to thank U. Bader, B. Bekka, T. Gelander, B. Nica, E. Ricard and the anonymous reviewers for valuable comments and interesting discussions on preliminary versions of this work.  Both authors wish to thank the Technion, where most of this work was achieved, for the good atmosphere and working conditions. 
\end{ack}

%
%

\section{Isometric group actions on $L_{p}(X,\mu)$-spaces}

In this section, we recall some general definitions and properties of linear and affine isometric actions on $L_{p}(X,\mu)$-spaces. Let $G$ be a topological group, and $(X,\mu)$ be a standard Borel measure space.

\subsection{Orthogonal representations on $L_{p}(X,\mu)$}

Let $1\leq p<\infty$, $p\neq2$, and denote by $O(L_{p}(X,\mu))$ the group of bijective linear isometries of the space $L_{p}(X,\mu)$. By a theorem of Banach and Lamperti, elements in $O(L_{p}(X,\mu))$ are the linear maps $U:L_{p}(X,\mu)\rightarrow L_{p}(X,\mu)$ described as follows :
$$Uf(x)=h(x)\esp (\frac{d\varphi*\mu}{d\mu}(x))^{1/p}\esp f(\varphi(x))  \esp\esp  (\ast)$$ 
where $h:X\rightarrow \mathbb{C}$ is a measurable function of modulus one, and $\varphi:X\rightarrow X$ is a $\mu$-class-preserving bijective transformation. In the particular case where $(X,\mu)$ is a discrete measure space, $\varphi$ is a permutation of the countable set $X$.\\

An orthogonal representation $\pi$ of the group $G$ on $L_{p}(X,\mu)$ is a group homomorphism $\pi:G\rightarrow O(L_{p}(X,\mu))$ such that the maps $g\mapsto\pi(g)f$ are continuous on $G$ for all $f\in L_{p}(X,\mu)$. For $1\leq p<\infty$ and $p\neq 2$, each element $\pi(g)$, $g\in G$, is given by formula $(\ast)$, and we call $\pi$ a (BL) representation. For $p=2$, the isometry group $O(L_{2}(X,\mu))$ contains much more elements than (BL) representations described by formula $(\ast)$ in the Banach-Lamperti theorem. Orthogonal representations on $L_{2}(X,\mu)$ which are (BL) representations will play a central role in our study.\\

In the current paper, the set of $\pi(G)$-invariant vectors is denoted by $L_{p}(X,\mu)^{\pi(G)}$. Recall that when $p=2$, the orthogonal complement $L_{2}'(X,\mu)$ to $L_{2}(X,\mu)^{\pi(G)}$ is also a $\pi(G)$ invariant space. Then the space $L_{2}(X,\mu)$ can be written as a direct sum of $\pi(G)$-invariant subspaces, $L_{2}(X,\mu) = L_{2}(X,\mu)^{\pi(G)} \oplus L_{2}'(X,\mu)$. The authors of \cite{BFGM} described a similar decomposition for $L_{p}(X,\mu)$, $p\neq2$, which we recall now.
 
For $p>1$, and $p'=p/(p-1)$, the contragradient representation $\pi^{*}:G\rightarrow O(L_{p'}(X,\mu))$ is defined by duality as follows:
$$<x,\pi^{*}(g)y>=<\pi(g^{-1})x,y> \textrm{ for all }g\in G,x\in L_{p}(X,\mu), y\in L_{p'}(X,\mu).$$
Then we define the following $\pi(G)$-invariant decomposition of the space $L_{p}(X,\mu)$ :
$$L_{p}(X,\mu)=L_{p}^{\pi(G)}\oplus L_{p}'(\pi)$$
where $L_{p}'(\pi)=\{\esp f\in L_{p}\esp\vert\esp \forall h\in L_{p'}^{\pi^{*}(G)}, <f,h>=0\esp\}$ is the annihilator of the $G$-invariant vectors for the contragradient representation $\pi^{*}:G\rightarrow O(L_{p'})$ of $\pi$ (more details can be found in Proposition 2.6 and section 2.c in \cite{BFGM} ). Given such a decomposition, one can consider the restriction $\pi':G\rightarrow O(L_{p}'(\pi))$. When $p=1$, the analog of the previous representation is the representation $\pi':G\rightarrow O(L_{1}/L_{1}^{\pi(G)})$, obtained from $\pi$ by composing with the quotient projection. If the context is clear, we will keep the notation $\pi$ in place of $\pi'$ in the sequel.\\

Let $\pi:G\rightarrow O(L_{p}(X,\mu))$ be a (BL) orthogonal representation. We will denote $\vert \pi\vert:G\rightarrow O(L_{p}(X,\mu))$ the map defined by the following formula, extended to $L_{p}(X,\mu)$ by linearity:
$$\vert \pi\vert(g)f=\vert \pi(g)f\vert \textrm{ for all }f\in L_{p}(X,\mu), f\geq 0.$$

Notice that $\vert\pi\vert$ is a (BL) representation, since $\pi$ is a (BL) representation. The representation $\vert \pi \vert$ is obtained from $\pi$ by simply omitting the function $h$ appearing in formula $(\ast)$ (or equivalently taking it to be constant 1) for each element $\pi(g)$, $g \in G$. We will say that $\pi$ is positive in the case where $\pi=\vert\pi\vert$. Also, we will say that $\pi$ is measure-preserving if the corresponding action of $G$ on $(X,\mu)$ is $\mu$-preserving, equivalently if the associated Radon-Nikodym derivative is equal to $1$ almost everywhere. By Banach's description of the isometries of $\ell_{p}$, all (BL) representations on the space $\ell_{p}$ are measure-preserving (see section 2 in \cite{BO} for details).

\subsection{The Mazur map}

A very useful tool to study $L_{p}$-spaces and their representations is the Mazur map. Here we recall the definition and some well-known properties of this map. The proofs of the basic properties of the Mazur map can be found in Chapter 9.1 of \cite{BL}. For $f\in L_{p}(X,\mu)$, we will use notations $sg(f)$ and $\vert f\vert$ for the polar decomposition components of $f$, so that $f=sg(f)\vert f\vert$ for $sg(f):X\rightarrow \mathbb{C}$ of modulus one, and $\vert f\vert\geq 0$.\\

Let $1\leq p,q<\infty$. The map
\begin{displaymath}
\begin{split}
M_{p,q}:&\esp L_{p}(X,\mu)\rightarrow L_{q}(X,\mu)\\
&\esp f=sg(f)\vert f\vert\mapsto sg(f)\vert f\vert^{p/q}
\end{split}
\end{displaymath}
is called the Mazur map. It induces a uniformly continuous homeomorphism between the unit spheres of $L_{p}(X,\mu)$ and $L_{q}(X,\mu)$. More precisely, the Mazur map $M_{p,q}$ is $\min(1,\frac{p}{q})$-H\"{o}lder on the unit sphere, i.e. for all $1\leq p, q<\infty$ :
$$\n M_{p,q}(x)-M_{p,q}(y)\n_{q}\leq C_{p,q}\n x-y\n_{p}^{\theta_{p,q}}\esp \esp (1)$$
for all $x,y\in S(L_{p}(X,\mu))$, where the constant $C_{p,q}$ depends only on $p,q$, and $\theta_{p,q}=\min(1,\frac{p}{q})$ (see Theorem 9.1 in \cite{BL} and \cite{R}). The following remark \ref{rem universal C} will be helpful for our proofs, and it relies on the property stated in the following lemma.

\begin{lem}
\label{lem-BL}
Let $Q$ be some compact interval in $[1,\infty[$. There exists a choice of constants $C_{p,q}>0$ such that the statements below hold.
\begin{enumerate}
\item For $q\in Q$ and $p\in [1,\infty[$ , formula (1) above holds with $C_{p,q}$.
\item For $p\in [1,\infty[$, the constant $C_{p} = \sup_{q\in Q} C_{p,q} >0$ satisfies $C_{p} < \infty$.
\item For $p\in [1,\infty[$, formula (1) holds with $C_{p}$ in place of $C_{p,q}$ for all $q\in Q$.
\end{enumerate}
\end{lem}
\begin{proof}
The proof is the one given in the proof of Theorem 9.1 in \cite{BL}, with some additional observation. So we detail only the part which requires a new argument. For $p>q$, the proof in \cite{BL} shows that $C_{p,q} = p/q$ so our statement is clear.\\
Authors of \cite{BL} also shows the existence of some $c_{p,q}>0$ such that the inequality below holds
$$\n M_{p,q}(x)-M_{p,q}(y)\n_{q}\geq c_{p,q} \times \n x-y\n_{p}^{\theta_{p,q}},$$
and it follows inequality (1) for $p<q$ since $M_{p,q}^{-1}=M_{q,p}$ and by taking $C_{p,q}=(c_{p,q})^{-1}$.\\
To prove our statement, it is sufficient to show that constant $c_{p,q}$ can be made independent of $q$ on compact subsets $Q\subset[1, \infty[$. Following the lines of the proof in \cite{BL}, we show that we have $\vert a^{p/q}-\theta b^{p/q}  \vert \geq c_{p} \vert a-\theta b\vert^{p/q}$ for $c_{p}$ independent of $q$ over compact subsets, and for all $a\geq b\geq 0$, and $\vert \theta\vert=1$. \\
First we notice that it is sufficient to prove the existence of $c_{p}$ for $b=1$, and for all $1\leq a \leq M$, where $M>0$ is some positive constant. Indeed, the inequality for $b=0$ holds with $c_{p}=1$. And for $b\neq 0$ we can consider $a/b$ in place of $a$ and $b=1$. As $a/b$ goes to infinity, $\vert (a/b)^{p/q}-\theta  \vert$ is equivalent to $\vert (a/b)-\theta \vert^{p/q}$. Hence, there exists some constant $M>0$ such that for all $0<b\leq a$ with $a/b\geq M$, we have $\vert (a/b)^{p/q}-\theta   \vert \geq (1/2) \times \vert (a/b)-\theta \vert^{p/q}$. Hence we can take $c_{p}=1/2$ for cases $b=0$, or $b=1$ and $a\geq M$.\\
Now we assume $b=1$, $1\leq a \leq M$ for some $M>0$, and we fix a compact subset $Q\subset [1, \infty[$. As noticed in \cite{BL}, the constant $c_{p}^{1}=1$ solves the case $\vert a-\theta\vert \leq 1$. 
For the case $\vert a-\theta\vert \geq 1$, we find a constant $c_{p}^{2}>0$ as follows. We notice that the set 
$$K = \{  \esp (q, a, \theta) \in Q \times[1,M] \times \mathbb{S}^{1} \esp\esp\vert\esp\esp  \vert a-\theta\vert \geq 1  \esp \} $$
is compact. Hence, the continuous map
\begin{displaymath}
\begin{split}
K  &\rightarrow \R^{+*}\\
(q, a, \theta) & \mapsto  \frac{\vert a^{p/q} - \theta  \vert}{\vert a -\theta \vert^{p/q}}
\end{split}
\end{displaymath}
attains a minimum $c_{p}^{2}>0$ on the compact set $K$. To finish the proof, we set $c_{p} = \min(1/2, c_{p}^{1}, c_{p}^{2})$. $\bsq$
\end{proof}

\begin {rem} \label {rem universal C} {\rm
Assume that we have a sequence $(p_{n})_{n}$ such that $p_{n} \in [1, \infty [$ for all $n$, and $\lim_{n} {p_{n}} =p $ for some $p\geq1$. Then by Lemma \ref{lem-BL} we can find a constant $C_{p}>0$, which does not depend on $n$, and such that:
$$\n M_{p,p_n}(x) - M_{p,p_n}(y)\n_{p_n}\leq C_{p}\n x-y\n_{p}^{\theta_{p,p_n}}$$ for all $x,y \in S(L_p(X,\mu)$. 
In particular, if we have 
$$\lim_n{ \n M_{p_n,p}(x_n)-M_{p_n,p}(y_n)\n_{p}} =0$$ for some $x_n ,y_n \in S(L_{p_n}(X,\mu))$, then the equality $\lim_n\n {x_n-y_n}\n_{p_{n}} =0$ holds as well.
}
\end{rem}
Let $x,y$ be any positive numbers and assume $1 \leq p \leq q$. Then we have the following inequality
$$\vert x^{p/q}-y^{p/q} \vert \leq \vert x-y\vert^{p/q}. $$
The following inequality is obtained from the inequality above:
$$\n M_{p,q}(x)-M_{p,q}(y)\n_{q} \leq \n x-y\n_{p}^{p/q} \textrm{ for all }x,y\geq 0 \textrm{ s.t. } x-y \in L_{p}(X,\mu) \esp\esp (2).$$
See also Proposition 1.1.4 in \cite{O2}.\\


Now let $p\geq1$ and let $\pi^{p}:G\rightarrow O(L_{p}(X,\mu))$ be a (BL) orthogonal representation of $G$ on $L_{p}(X,\mu)$. Since $\pi^{p}$ is described by equation $(\ast)$, the formula
$$\pi^{q}(g)x=(M_{p,q}\circ \pi^{p}(g)\circ M_{q,p})x \textrm{ for all }g\in G, x\in L_{q}(X,\mu)\esp\esp(3)$$
define an orthogonal representation $\pi^{q}:G\rightarrow O(L_{q}(X,\mu))$.
Note that when the associated action on $(X,\mu)$ is measure preserving the representations $\pi^{p}$ and $\pi^q$ coincide on the intersection, $L_{p}(X,\mu) \cap L_{q}(X,\mu)$ \\

A sequence $(f_{n})_{n}$ in a Banach space $F$ is said to be a sequence of almost invariant vectors for an orthogonal representation $\pi:G\rightarrow O(F)$ if $\n f_{n}\n_{F}=1$ and 
$$\lim_{n}\n \pi(g)f_{n}-f_{n}\n_{F}=0\textrm{  uniformly on compact subsets of }G. $$
The following lemma was first stated in the proof of Theorem A in \cite{BFGM} (see Remark 4.3 in \cite{BFGM}). For the convenience of the reader, we recall the proof of this lemma, given in Proposition 2.5.1 of \cite{O2} in a more general setting.
\begin {lem} \label {T and Tlp}For a family $(\pi^{p})_{1\leq p<\infty}$ of (BL) orthogonal representations, and for all $1\leq p<\infty$, $1\leq q<\infty$, a group $G$ admits a sequence of almost invariant vectors for the representation $(\pi^{p})'$ if and only if it admits a sequence of almost invariant vectors for $(\pi^{q})'$\end{lem}
\begin{proof}
Let $f_n\in S(L_{p}'(\pi^{p}))$ be a sequence of almost invariant vectors. Note that for that the following inequalities hold (see Proposition 3.5 in \cite{O1}) : 
$$d(f_{n}, L_{p}^{\pi^{p(G)}})\geq1/2\textrm{ for all }n.$$
Notice that $L_{q}^{\pi^{q}(G)}=M_{p,q}(L_{p}^{\pi^{p}(G)})$. Define $ h_n = M_{p,q}f_n$. By the uniform continuity of the Mazur map then, $h_n$ is a sequence of unit vectors in $L_{q}$ with
$$d(h_{n},L_{q}(\pi^{q}))\geq \delta>0.$$
Furthermore we have for every compact set $Q \subset G$,
$$\lim_{n}\sup_{g\in Q}\n h_{n}-\pi^{q}(g)h_{n}\n_{q}=0.$$
Consider first the case $q>1$. Let $v_{n}$ denote the projections of $h_{n}$ on the complement subspace $L_{q}'(\pi^{q})$. Since $\n v_{n}\n_{q}\geq \delta$, the uniform convergence on compact subsets holds when replacing $h_{n}$ by $v_{n}/\n v_{n}\n_{p}$. Hence $(v_{n}/\n v_{n}\n_{p})_{n}$ is a sequence of almost invariant vectors for $\pi^{q}$ with values in $(\pi^{q})'$. \\
For $q=1$, we consider $v_{n}$ the projection of $h_{n}$ on the quotient space $F=L_{1}/L_{1}^{\pi^1(G)}$. Then $\n v_{n}\n_{F}\geq \delta$, and the sequence $v_{n}/\n v_{n}\n_{F}$ is a sequence of almost invariant vectors for the representation $(\pi^{1})':G\rightarrow O(F)$.
$\bsq$
\end{proof}

\subsection{Affine isometric actions on $L_{p}(X,\mu)$}

Denote by ${\rm Isom}(L_{p}(X,\mu))$ the group of affine bijective isometries of $L_{p}(X,\mu)$. An affine isometric action of $G$ on $L_{p}(X,\mu)$ is a group homomorphism $\alpha:G\rightarrow {\rm Isom}(L_{p}(X,\mu))$ such that the maps $g\mapsto \alpha(g)f$ are continuous for all $f\in L_{p}(X,\mu)$. Then we have
$$\alpha(g)f=\pi(g)f+b(g)\textrm{ for all }g\in G, f\in L_{p}(X,\mu) $$
where $\pi:G\rightarrow O(L_{p}(X,\mu))$ is an orthogonal representation, and $b:G\rightarrow L_{p}(X,\mu)$ is a 1-cocycle associated to $\pi$, that is a continuous map satisfying the following relations :
$$b(gh)=b(g)+\pi(g)b(h)\textrm{ for all }g,h\in G.$$
Given an orthogonal representation $\pi$ and a cocycle associated to $\pi$, we will sometimes use the notation $\alpha=(\pi,b)$ to denote the affine representation whose linear part is $\pi$, and translation part is $b$. We denote by $\H1(G,\pi)$ the first cohomology group with coefficients in $\pi$, that is the quotient of the space of $1$-cocycles associated to $\pi$, by the subspace of $1$-coboundaries (a coboundary is a cocycle of the form $b(g)=f-\pi(g)f$ for some $f\in L_{p}(X,\mu)$). \\

We recall that an affine isometric action of $G$ on $L_{p}(X,\mu)$ has a fixed-point if and only if the associated cocycle $b:G\rightarrow L_{p}(X,\mu)$ is a bounded map (see Lemma 2.14 in \cite{BFGM} for $p>1$, and \cite{BGM} for $p=1$). The action is said to be proper if $\lim_{g\rightarrow\infty}\n b(g)\n_{p}=\infty$. A topological group $G$ is said to have the fixed-point property $(F_{L_{p}(X,\mu)})$ if every action by affine isometries of $G$ on $L_{p}(X,\mu)$ admits a fixed-point, that is $\H1(G,\pi)=\{0\}$ for all orthogonal representation $\pi:G\rightarrow O(L_{p}(X,\mu))$. \\

A locally compact second countable group $G$ is said to have property $(T_{L_{p}(X,\mu)})$ if for any orthogonal representation $\pi:G\rightarrow O(L_{p}(X,\mu))$, the restriction of $\pi$ on $L_{p}'(\pi)$ (when $p>1$) or $L_{1}/L_{1}^{\pi(G)}$ (when $p=1$), has no sequence of almost invariant vectors. The following facts were proved by Guichardet for the case of Hilbert spaces. Later they were proved in the general setting of $L_{p}$ spaces in \cite{BFGM} (see section 3.a for proofs):
\begin {prp} \label{Guichardet} If an orthogonal representation $\pi:G\rightarrow O(L_{p}(X,\mu))$ of a second countable group $G$ has a sequence of almost invariant vectors in $L_{p}'(\pi)$, then $\H1(G,\pi)\neq\{0\}$.
\end{prp}
An important consequence of the previous proposition is the following corollary, whose result holds in a much more general context (see Theorem 1.3 in \cite{BFGM}).
\begin {cor} Property $(F_{L_{p}(X,\mu)})$ implies property $(T_{L_{p}(X,\mu)})$.
\end{cor}
In the case where the group $G$ is second countable and generated by a compact subset $Q\subset G$, an orthogonal representation $\pi:G\rightarrow O(L_{p}(X,\mu))$ without any sequence of almost invariant vectors in $L_{p}(X,\mu)'$ satisfies the condition: there exists $\epsilon>0$ such that for all $x\in L_{p}(X,\mu)' $ (for $p>1$), or all $x\in L_{1}(X,\mu)/L_{1}^{\pi(G)}(X,\mu)$ (for $p=1$), the following inequality holds:   
$$\epsilon \n x\n_{p}\leq \n x-\pi(g)x\n_{p}  \textrm{ for some }g\in Q.$$
When $p=2$ and the previous condition holds for all orthogonal representations $\pi:G\rightarrow O(L_{2}(X,\mu))$, the pair $(Q,\epsilon)$ is called a Kazhdan pair.

%
%

\section{About the closedness of the fixed-point spectrum}

Our goal is to show that the fixed-point spectrum $\mathcal{F}_{L_{\infty}(X,\mu)}$ is an interval containing 1. To achieve this, we study the topological properties of the set $\mathcal{F}_{L_{\infty}(X,\mu)}$ in $[1,\infty[$. More precisely our goal is to show it is connected. In this section, we show closedness results for $\mathcal{F}_{L_{\infty}(X,\mu)}(G,\pi)$. To prove these results, we use a limit-version of a well-known fact about almost invariant vectors for (BL) representations, which is interesting in its own right (Proposition \ref{prp-aiv} below). \\

We begin with some general facts concerning isometric actions on $L_{p}$-spaces.

\begin{lem}\label{lem-distance}
Let $1\leq p<\infty$, $1\leq p_{n}<\infty$, $p_{n}\neq1$, be such that $\lim_{n}p_{n}=p$. Let $G$ be a topological group. Let $(X,\mu)$ be a standard Borel measure space. Let $\pi^{p}:G\rightarrow O(L_{p}(X,\mu))$ be a (BL) orthogonal representation. Then there exists $C'>0$ and $N$ such that for all $n\geq N$, we have
$$d(M_{p_{n},p}f, L_{p}(X,\mu)^{\pi^{p}(G)})\geq C'\textrm{ for all }f\in S(L_{p_{n}}(X,\mu)'(\pi^{p_{n}})).$$
\end{lem}

\begin{proof}
For $f\in S(L_{p_{n}}'(\pi^{p_{n}}))$, the following inequalities hold (see Proposition 3.5 in \cite{O1}) : 
$$d(f, L_{p_{n}}^{\pi^{p_{n}(G)}})\geq1/2\textrm{ for all }n.$$
Notice that $L_{p}^{\pi^{p}(G)}=M_{p_{n},p}(L_{p_{n}}^{\pi^{p_{n}}(G)})$. Assume by contradiction, that the result does not hold. Then there exist $f_{n}\in S(L_{p_{n}}'(\pi^{p_{n}}))$ and $a_{n}\in L_{p_{n}}^{\pi^{p_{n}}(G)}$ such that
$$\lim_{n} \n M_{p_{n},p}f_{n} - M_{p_{n},p}a_{n} \n_{p}=0.  \esp\esp (3)$$
In particular, since $\n a_n \n_{p_{n}} =\n M_{p_n,p}a_n\n_p$, we have $\lim_{n} \n a_n \n_{p_{n}} =1$. Let $\overline{a_{n}} = \frac{a_{n}}{\n a_{n}\n_{p_{n}}}$. It follows that 
\begin {displaymath}
\begin{split}
\lim_{n} \n M_{p_{n},p}f_{n} - M_{p_{n},p}\overline{a_{n}} \n_{p} \leq  \lim_{n} \esp ( \esp\n M_{p_{n},p}f_{n} & - M_{p_{n},p} a_n\n_{p} \\ & +\n M_{p_{n},p} a_n - M_{p_{n},p} \overline{a_n}\n_{p} \esp). 
\end{split}
\end{displaymath}
As $\n a_{n}\n_{p_{n}} \rightarrow 1$, we conclude from (3) that: 
$$ \lim_{n} \n M_{p_{n},p} f_n -M_{p_{n},p}\overline{a_{n}}\n_{p}=0.$$
Applying Remark \ref{rem universal C} we obtain
$$\lim_{n} \n f_{n} - \overline{a_{n}} \n_{p_{n}} = 0. $$
As $f_n \in S(L_{p_{n}}'(\pi^{p_{n}}))$, and $ \overline{a_n}\in L_{p_{n}}^{\pi^{p_{n}(G)}}$ this contradicts the fact that
$$d(f_{n}, L_{p_{n}}^{\pi^{p_{n}(G)}})\geq1/2\textrm{ for all }n.$$
$\bsq$
\end{proof}

The following proposition is a limit version of Proposition \ref{T and Tlp}.
\begin{prp}\label{prp-aiv}
Let $1\leq p<\infty$, $1\leq p_{n}<\infty$, $p_{n}\neq1$, be such that $\lim_{n}p_{n}=p$. Let $G$ be a second countable topological group. Let $(X,\mu)$ be a standard measure space. Let $\pi^{p}:G\rightarrow O(L_{p}(X,\mu))$ be a (BL) orthogonal representation. Assume that there exists a sequence $(f_{n})_n \in S(L_{p_{n}}(X,\mu)'(\pi^{p_{n}}))$ such that
$$\lim_{n}\sup_{g\in Q}\n f_{n}-\pi^{p_{n}}(g)f_{n}\n_{p_{n}}=0$$
for all compact subsets $Q\subset G$. Then there exists a sequence of almost invariant vectors for $\pi^{p}$ in $L_{p}(X,\mu)'(\pi^{p})$.
\end{prp}

\begin{proof}
Let $Q\subset G$ be a compact set. Define $h_{n}=M_{p_{n},p}f_{n}\in S(L_{p}(X,\mu))$ for $n\in\mathbb{N}$. As in the previous lemma, we can choose $C$ for the estimates of the Mazur maps $M_{p_{n},p}$, independent of $n$. Then we have
$$\lim_{n}\sup_{g\in Q}\n h_{n}-\pi^{p}(g)h_{n}\n_{p}=0.\esp\esp(\ast)$$
Now by Lemma \ref{lem-distance}, there exists $C'>0$ such that $d(h_{n}, L_{p}^{\pi^{p}(G)})\geq C'$ holds for all $n$ large enough.\\
Consider first the case where $p>1$. Let $v_{n}$ denote the projections of $h_{n}$ on the complement subspace $L_{p}'(\pi^{p})$. Since $\n v_{n}\n_{p}\geq C'$ for $n$ large enough, the uniform convergence $(\ast)$ on compact subsets holds when replacing $h_{n}$ by $v_{n}/\n v_{n}\n_{p}$. Hence $(v_{n}/\n v_{n}\n_{p})_{n}$ is a sequence of almost invariant vectors for $\pi^{p}$ with values in $L_{p}'(\pi^{p})$. \\
For $p=1$, we consider $v_{n}$ the projection of $h_{n}$ on the quotient space $F=L_{1}/L_{1}^{\pi^{p}(G)}$. Then $\n v_{n}\n_{F}\geq C'$, and the sequence $v_{n}/\n v_{n}\n_{F}$ is a sequence of almost invariant vectors for the representation $(\pi^{1})':G\rightarrow O(F)$.
$\bsq$
\end{proof}

\medskip

Now we show a closedness property for some sets $\mathcal{F}_{L_{\infty}(X,\mu)}(G,\pi)$. Our proof  requires two important assumptions : the monoticity of $(L_{p}(X,\mu))_{p}$ for the inclusion, and the representation $\pi$ to be measure-preserving. 

\begin{prp}\label{prp-closedness}
Let $1\leq p<\infty$, $1\leq p_{n}<\infty$, $p_{n}\neq1$, be such that $\lim_{n}p_{n}=p$. Assume that $L_{p}(X,\mu)\subset L_{p_{n}}(X,\mu)$ for all $n$, and that $\lim_{n}\n f\n_{p_{n}}=\n f\n_{p}$ for all $f\in L_{p}(X,\mu)$. Let $G$ be a second countable topological group. Let $\pi^{p}:G\rightarrow O(L_{p}(X,\mu))$ be a (BL) measure-preserving orthogonal representation. Assume also that
$$\H1(G,\pi^{p_{n}})=\{0\}.$$
Then we have 
$$\H1(G,\pi^{p})=\{0\}.$$
\end{prp}

\begin{proof}
Let $b:G\rightarrow L_{p}(X,\mu)$ be a cocycle associated to $\pi^{p}$. Let $n\in\mathbb{N}$. Since by assumption $L_{p}(X,\mu)\subset L_{p_{n}}(X,\mu)$ and $\pi^{p}$ is measure-preserving, $b$ defines also a cocycle for the representation $\pi^{p_{n}}$. Then there exists $f_{n}\in L_{p_{n}}(X,\mu)$ such that 
$$b(g)=f_{n}-\pi^{p_{n}}(g)f_{n}\textrm{ for all }g\in G.$$
Notice that we can assume that $f_{n}\in L_{p_{n}}'(\pi^{p_{n}})$ for all $n$, without loss of generality. 

\begin{clm}\label{claim-fn-bounded} The sequence $(\n f_{n}\n_{p_{n}})_{n}$ is bounded. \end{clm}

Let us prove this claim now. We assume the contrary and will show that $(\pi^{p}$)' has a sequence of almost invariant vectors, which is not possible by assumption. Assume then that there exist a subsequence of $(f_n)_n$, (and use the same notation for the subsequence), with  $\lim_n{(\n f_{n}\n_{p_{n}})_{n}} = \infty$  Let $Q\subset G$ be a compact subset. Since $\lim_{n}\n b(g)\n_{p_{n}}=\n b(g)\n_{p}$ for all $g\in G$, for $n$ large enough we have
\begin{displaymath}
\begin{split}
\sup_{q\in Q}\n \frac{f_{n}}{\n f_{n} \n_{p_{n}}} -\pi^{p_{n}}(g)   \frac{f_{n}}{\n f_{n} \n_{p_{n}}}\n_{p_{n}} & =   \frac{\sup_{g\in Q} \n b(g) \n_{p_{n}} }{ \n f_{n} \n_{p_{n}}  }  \\
&\leq 2  \frac{\sup_{g\in Q} \n b(g) \n_{p} }{ \n f_{n} \n_{p_{n}}  }
\end{split}
\end{displaymath}
and the right-hand side of the inequality tends to $0$ as $n$ tends to $\infty$. Thus we can apply Proposition \ref{prp-aiv} on the sequence $\frac{f_{n}}{\n f_n \n_{p_n}}$  and deduce that there exists a sequence of almost invariant vectors for $(\pi^{p})'$. Hence by Lemma \ref{T and Tlp} the representations $(\pi^{p_{n}})'$ admits a sequence of almost invariant vectors for all $n$ as well. This contradicts the vanishing of $\H1(G,\pi^{p_{n}})$ (by the $L_p$ version of Guichardet's Theorem, stated in Proposition \ref{Guichardet}), and Claim \ref{claim-fn-bounded} is proved. \\

Then there exists $C''>0$ such that $\n f_{n}\n_{p_{n}}\leq C''$ for all $n\in\mathbb{N}$. Let $g\in G$. We have
$$\n b(g)\n_{p}=\lim_{n}\n b(g) \n_{p_{n}}\leq 2C''.$$
Hence the cocycle $b:G\rightarrow L_{p}(X,\mu)$ is bounded in $L_{p}(X,\mu)$. $\bsq$
\end{proof}

\medskip

%
%

\section{The fixed-point spectrum for actions on $\lp$}

This section handles the case of actions of $G$ on $\lp$-spaces over a discrete space $X$. It is divided in two parts. In the first part, we prove Theorem \ref{thm-main-lp}. Some elements in the proofs of our Proposition \ref{prp-lp} and Lemma \ref{lem-lp} can be found in section 3.1 of \cite{C}, and in particular the crucial use of estimates (2) in Section 2.2. However, the approach of the current paper requires different statements from those in \cite{C}, and we provide relevant results with complete proofs for convenience of the reader. In the second part of the section, we discuss property $(F_{\lp})$ and exhibit examples of groups $G$ illustrating each form of possible fixed-point spectra $\mathcal{F}_{\ell_{\infty}}(G)$ listed in the statement of Theorem \ref{thm-main-lp}.

\subsection{Proof of Theorem \ref{thm-main-lp}}

Theorem \ref{thm-main-lp} is a direct consequence of Proposition \ref{prp-lp}, which we state and prove after a few useful remarks.\\

Let $\pi:G\rightarrow O(\lp)$ be a $(BL)$ orthogonal representation. Write $\lp=\lp(X)$ and decompose $X=X^{f} \cup{X^{i}}$ where $X^{f}$ is the union of the finite orbits of the $G$-action, and $X^{i}$ the union of the infinite orbits. In the sequel, we will make use of the decomposition of $\pi$ as $\pi=\pi^{f}\oplus\pi^{i}$ with respect to the decomposition $\lp(X)=\lp(X^{f})\oplus\lp(X^{i})$.\\

We make the following observation :
$$H^{1}(G,\pi)=\{0\}\esp  \Leftrightarrow  \esp H^{1}(G,\pi^{f})=\{0\} \textrm{ and }H^{1}(G,\pi^{i})=\{0\}.$$
Hence $\mathcal{F}_{\ell_{\infty}}(G,\pi)$ is an interval of the form $[1,p_{c}]$ or $[1,p_{c}[$ for some $p_{c}\geq1$ if the two following conditions hold:\\
(i) $\mathcal{F}_{\ell_{\infty}}(G,\pi^{f})$ is an interval $[1,p_{c}']$ or $[1,p_{c}'[$ for some $p_{c}'\geq1$;\\
(ii) $\mathcal{F}_{\ell_{\infty}}(G,\pi^{i})$ is an interval $[1,p_{c}'']$ or $[1,p_{c}''[$ for some $p_{c}''\geq1$.\\

For technical reasons, we are not able to prove that $\mathcal{F}_{\ell_{\infty}}(G,\pi)$ is an interval for all representations $\pi$ (but for positive representations, we do). Nevertheless, the following Proposition \ref{prp-lp} shows a very close result to the connectedness of all sets $\mathcal{F}_{\ell_{\infty}}(G,\pi)$, and Theorem \ref{thm-main-lp} is a straightforward consequence of this result. Notice that when a group $G$ is not compactly generated, then $\mathcal{F}_{\ell_{\infty}}(G)$ is empty: for any $p\geq1$, such a group does not have property $(T_{\lp})$ (see \cite{BO}), nor property $(F_{\lp})$ by well-known results.
\begin{prp}\label{prp-lp}
Let $G$ be a second countable group generated by a compact set $Q$. Let $\pi:G\rightarrow O(\lp)$ be a (BL) orthogonal representation. Let $\pi^{f},\pi^{i}$ be defined as in the previous discussion. Then we have:\\
(i) $\mathcal{F}_{\ell_{\infty}}(G,\pi^{f})=[1,p_{c}']$ or $[1,p_{c}'[$;\\
(ii) $\mathcal{F}_{\ell_{\infty}}(G,\pi^{i})\cap\mathcal{F}_{\ell_{\infty}}(G,\vert\pi^{i}\vert)=[1,p_{c}'']$ or $[1,p_{c}''[$.
\end{prp}

Before giving the proof of Proposition \ref{prp-lp}, we start with a simple lemma.

\begin{lem}\label{lem-lp}
Let $1\leq p<q$. Let $\pi^{p}:G\rightarrow O(\lp)$ be a (BL) orthogonal representation. Assume that $X^{f}=\emptyset$, and that
$$H^{1}(G,\pi^{q})=\{0\} \textrm{ and } H^{1}(G,\vert\pi^{p}\vert)=\{0\}.$$
Then we also have
$$ H^{1}(G,\pi^{p})=\{0\}. $$
\end{lem}

\begin{proof}
Let $b:G\rightarrow\lp$ be a cocycle for $\pi$. As $\lp \subset \ell_{q}$, $b$ is also a cocycle for $\pi^{q}$. By assumption, there exists $x\in\ell_{q}$ such that
$$b(g)=x-\pi(g)x \textrm{ for all }g\in G.$$
For all $a,b\in\mathbb{C}$, we have  
$$\vert \vert a\vert-\vert b\vert \vert\leq \vert a-b\vert,$$
and the following inequalities follow for all $g\in G$ :
$$ \n \vert x\vert - \vert\pi^{p}\vert(g)\vert x\vert\n_{p} \leq \n b(g)\n_{p}. $$
Hence the formula $c(g)= \vert x\vert - \vert\pi^{p}\vert(g)\vert x\vert$ defines a cocycle for $\vert\pi^{p}\vert$ with values in $\lp$. By assumption, there exists $y\in\lp$ such that
$$ \vert x\vert - \vert\pi^{p}\vert(g)\vert x\vert = y- \vert\pi^{p}\vert(g)y \textrm{ for all }g\in G.$$
It follows then, that $\vert x\vert -y $ is a $ \vert \pi^{p} \vert$ invariant vector.  Since $\lp \subset \ell_{q}$, and $\vert\pi^{p}\vert$ coincides with $\vert\pi^{q}\vert$ on $\lp$, this is a $ \vert \pi^{q} \vert$ invariant vector as well. From the assumption $X^{f}=\emptyset$, we deduce that $\vert \pi^{q}\vert$ does not have non-zero invariant vectors. Hence $\vert x\vert -y =0$, i.e. $\vert x\vert=y$, and then $x\in\lp$. $\bsq$ 
\end{proof}
\begin {rem}

Note that in $\lp(X^{i})$ there are no non-zero invariant vectors for both $\pi$, and $\vert \pi\vert$.
\end {rem}

\medskip
Now we are able to prove Proposition \ref{prp-lp}.

\begin{proof-prp12} (i) We will show that for all $1<p<q$, $H^{1}(G,(\pi^{f})^{q})=\{0\}$ implies $H^{1}(G,(\pi^{f})^{p})=\{0\}$. If this assertion holds, the set $\mathcal{F}_{\ell_{\infty}}(G,\pi^{f})$ is an interval of one the following forms: $[1,p_{c}']$, $[1,p_{c}'[$, $]1,p_{c}']$ or $]1,p_{c}']$ for some $p_{c}'\geq1$. By the closedness property of Proposition \ref{prp-closedness}, $\mathcal{F}_{\ell_{\infty}}(G,\pi^{f})$ is an interval of the form $[1,p_{c}']$ or $[1,p_{c}'[$.\\

Let $b:G\rightarrow\lp(X^{f})$ be a cocycle associated to $\pi^{f}$. By assumption, there exists $x\in\ell_{q}$ such that
$$b(g)=x-\pi^{f}(g)x\textrm{ for all }g\in G.$$
Notice that we can assume that $x\in \ell_{q}(X^{f})'$ without loss of generality. We decompose $X^{f}=\sqcup_{j\geq 0} X_{j}$ in (finite) orbits, and write $x=\oplus_{j} x_{j}$ in $\ell_{q}(X^{f})=\oplus_{j}\ell_{q}(X_{j})$. From the definition of the complement $\ell_{q}'$ (see section 2.1), it is clear that $x_{j}\in\ell_{q}(X_{j})'$ for all $j$.\\
Since $H^{1}(G,(\pi^{f})^{q})=\{0\}$, $\pi^{f}$ has no sequence of almost invariant vectors in $\ell_{q}(X^{f})'$. As recalled in Lemma \ref {T and Tlp}, it is then also true that $\pi^{f}$ has no sequence of almost invariant vectors in $\ell_{p}(X^{f})'$. In particular,
there exists $\epsilon>0$ such that 
$$\epsilon\n u\n_{p}\leq \sup_{g\in Q} \n u-\pi^{f}(g)u\n_{p} \textrm{ for all }u\in \lp'(\pi^{f}).$$
Define $u_{n}=\oplus_{j\leq n}x_{j}\in\ell_{q}'$, and $M=\sup_{g\in Q}\n b(g)\n_{p}$. Then we have for all $n$,
\begin{displaymath}
\begin{split}
\epsilon \n u_{n}\n_{p}& \leq\sup_{g\in Q} \n u_{n}-\pi^{f}(g)u_{n}\n_{p}\\
& \leq \sup_{g\in Q}\n b(g)\n_{p}\\
&= M.
\end{split}
\end{displaymath}
Hence the sequence $(u_{n})_{n}$ is bounded in $\lp$ by $M/\epsilon$. Hence $x\in \lp$ and the proposition is proved.\\


(ii) Let $1\leq p<q<\infty$. We assume that $q\in\mathcal{F}_{\ell_{\infty}}(G,\pi^{i})\cap\mathcal{F}_{\ell_{\infty}}(G,\vert\pi^{i}\vert)$. Then we need to show that $p\in\mathcal{F}_{\ell_{\infty}}(G,\pi^{i})\cap\mathcal{F}_{\ell_{\infty}}(G,\vert\pi^{i}\vert)$. By Lemma \ref{lem-lp}, it is sufficient to show that $H^{1}(G,\vert((\pi^{i})^{p})\vert)=\{0\}$.\\

Let $b:G\rightarrow\lp(X^{i})$ be a cocycle associated to $\vert((\pi^{i})^{p})\vert$. By assumption, there exists $x\in\ell_{q}$ such that
$$b(g)=x-\vert\pi^{i}\vert(g)x \textrm{ for all }g\in G.$$
On the one hand, we have for all $g\in G$,
$$ \n \vert x\vert-\vert\pi^{i}\vert(g)\vert x\vert \n_{p} \leq \n x-\vert\pi^{i}\vert(g)x \n_{p}.$$
On the other hand, estimates (2) for the Mazur maps in section 2.2 imply the following inequalities for all $g\in G$ :
$$ \n M_{p,q}( \vert x\vert )-\vert\pi^{i}\vert(g)M_{p,q}( \vert x\vert ) \n_{q} \leq \n  \vert x\vert-\vert\pi^{i}\vert(g)\vert x\vert \n_{p}^{p/q}.$$
Hence the formula $M_{p,q}( \vert x\vert )-\vert\pi^{i}\vert(g)M_{p,q}( \vert x\vert )$, $g\in G$, define a cocycle associated to $\vert\pi^{i}\vert$ with values in $\ell_{q}$. By assumption, there exists $y\in\ell_{q}$ such that
$$M_{p,q}( \vert x\vert )-\vert\pi^{i}\vert(g)M_{p,q}( \vert x\vert )=y-\vert\pi^{i}\vert(g)y \textrm{ for all }g\in G.$$
It follows then, that $M_{p,q} \vert x\vert -y $ is a $ \vert \pi^{i} \vert$ invariant vector. Since $X^{i}$ has only infinite orbits for the $G$-action, there is no non-zero $\vert \pi^{i}\vert(G)$-invariant vector in $\ell_{q}$. Hence we have $\vert x\vert=M_{q,p}y$. It follows that $x\in\lp$ and the proposition is proved. $\bsq$
\end{proof-prp12}

As a particular case of the previous proofs, we have the following result.
\begin{cor}\label{cor-fixedSpectrumPi}
Let $G$ be a second countable group, and $p\geq1$. Let $\pi:G\rightarrow O(\lp)$ be a (BL) orthogonal positive representation. Then $\mathcal{F}_{\ell_{\infty}}(G,\pi)$ is empty, or is an interval of the form $[1,p_{c}[$ ($p_{c}\leq\infty$) or $[1,p_{c}]$ ($p_{c}<\infty$).
\end{cor}

\subsection{More about property $(F_{\lp})$}

Now we give examples of groups $G$ for which the fixed-point spectrum $\mathcal{F}_{\ell_{\infty}}(G)$ is the union of two intervals $[1,\pc[\backslash\{2\}$ ($\pc>2$). We recall the following well-known relationships between properties $(T)$, $\Tlp$ and $\Flp$ (see for example Theorem 1.3 in \cite{BFGM}):
\begin{displaymath}
\begin{split}
&\Flp\Rightarrow \Tlp\esp\esp\esp\esp\esp\esp\esp\esp\esp\textrm{ for all }p\geq1,\\
&(T)\Rightarrow\Flp\Rightarrow\Tlp \textrm{ for }1\leq p\leq2.\\
\end{split}
\end{displaymath}

There are naturally few continuous actions of a connected group on discrete spaces. As shown in \cite{BO} where property $(T_{\ell_{p}})$ was studied, the class of connected groups having property $\Tlp$ is restricted to the groups with compact abelianization (see Corollary 3 in \cite{BO}). We now show that the latter groups also have property $\Flp$.

\begin{thm}\label{thm-connected-lp}
Let $G$ be a locally compact second countable group. Assume that $G$ is connected. Let $p \ne 2$. Then the following assertions are equivalent:\\
(i) $G$ has property $\Flp$;\\
(ii) $G$ has property $\Tlp$;\\
(iii) the abelianised group $G/\overline{[G,G]}$ is compact.
\end{thm}
\begin {rem}{\rm
For $p=2$, property $(T)$ is equivalent to property $(F_{\ell_{2}})$ (also called property $(FH)$). Any group with property $(T)$ has compact abelianised group, but it is no longer true for non-Kazhdan groups (such as $SL_2(\mathbb{R})$). For $p\geq 1$ (possibly $p=2$), the following proof shows that any $(BL)$ representation $\pi$ of a connected group $G$ on $\ell_p$ ($G$ possibly being non-Kazhdan), whose abelianised group is compact, satisfies $\H1(G,\pi)=\{0\}$.
}
\end{rem}
\begin{proof}
We only have to show the implication $(iii)\Rightarrow(i)$. Since $G$ is connected, the orbits of a continuous action of $G$ on a countable infinite discrete set $X$ are singletons. Hence a (BL) representation on $\ell_{p}(X)$ is a direct sum of continuous unitary characters. So we have to show that $\H1(G,\pi)=\{0\}$ for every orthogonal representation $\pi:G\rightarrow O(\ell_{p})$ of the form
$$\pi=\bigoplus_{i\in I}\chi_{i}$$
where every $\chi_{i}$ is a continuous character on $G$. Let $\pi=\bigoplus_{i\in I}\chi_{i}$ be such a representation. Let $H=\bigcap_{i\in I}{\rm Ker}(\chi_{i})$ be the kernel of the homomorphism
\begin{displaymath}
\begin{split}
\varphi:\textrm{   }&G\rightarrow\prod_{i\in I}S^{1}\\
&g\rightarrow(\chi_{i}(g))_{i\in I}.
\end{split}
\end{displaymath}
Denote by $N=\overline{[G,G]}$ and $p:G\rightarrow G/N$ the quotient projection. Notice that, since $N\subset H$, $\pi(n)=id$ for all $n\in N$ and $\ell_{p}^{\pi(N)}=\ell_{p}$. So $\pi$ factors through $G/N$ as $\pi=\rho\circ p$. As a consequence of the Hoschild-Serre spectral sequence, we have the following exact sequence :
$$\H1(G/N,\rho)\rightarrow \H1(G,\pi)\rightarrow \H1(N,1_{N})^{G/N}.$$
On the one hand, we have $\H1(G/N,\rho)=\{0\}$ since $G/N$ is compact. On the other hand, $\ell_{p}$ is commutative and $N=\overline{[G,G]}$, so it follows $\H1(N,1_{N})={\rm Hom}(N,\ell_{p})=\{0\}$. As a consequence, we have $\H1(G,\pi)=\{0\}$ and this finishes the proof. $\bsq$
\end{proof}

\begin{rem}{\rm
(i) The proof of Theorem \ref{thm-connected-lp} was inspired by techniques used in \cite{BMV} to obtain results concerning the cohomology (and the reduced cohomology) associated to the regular representation. In particular, it is shown that the reduced $\ell_{p}$-cohomology $\overline{H}_{p}^{1}(\Gamma,\lambda_{\Gamma})$ vanishes if and only if $p\leq e(\Gamma)$ for some lattices $\Gamma\subset G$ in rank one groups $G$ and some critical exponent $e(\Gamma)$ explicitly defined (see Theorem 2 and the discussion which follows its statement). Hence the fixed-spectrum for reduced cohomology $\{\esp p\geq 1 \esp\vert\esp \overline{H}_{p}^{1}(\Gamma,\lambda_{\Gamma})= \{0\} \esp\}$ is a closed interval of the form $[1,e(\Gamma)]$. By analogy, our Corollary \ref{cor-fixedSpectrumPi} asserts that the fixed-point spectrum $\mathcal{F}_{\ell_{\infty}}(\Gamma,\lambda_{\Gamma})=\{\esp p\geq 1 \esp\vert\esp H_{p}^{1}(\Gamma,\lambda_{\Gamma})= \{0\} \esp\}$ is an interval of the form $[1,p_{c}[$ or $[1,p_{c}]$. Results providing relationships between $e(\Gamma)$ and $p_{c}$ could possibly help to decide if $\mathcal{F}_{\ell_{\infty}}(\Gamma,\lambda_{\Gamma})$ is open or closed. \\
(ii) Examples which suggest that the set $\mathcal{F}_{\ell_{\infty}}(G,\pi)$ is closed can be found in \cite{B} ( see Remark 4 in section 1.6, and Remark 3 in section 2.4 ).
}
\end{rem}

\begin{q}{\rm
Does $\pc$ always belong to $\mathcal{F}_{\ell_{\infty}}(G,\pi)$ when the fixed-point spectrum is non-empty and $\pi$ is a positive (BL) representation ?
}
\end{q}

The following examples show that the fixed-point spectrum can have one or two connected components when it is not empty.

\begin{examples}\label{ex-SL2}{\rm
(i) For instance, the group $SL_{2}(\mathbb{R})$ does not have property $(T)$, but has property $\Flp$ for all $p\neq2$ by the previous theorem. So we have $\mathcal{F}_{\ell_{\infty}}(SL_{2}(\mathbb{R})=[1,\infty[\backslash\{2\}$.\\
(ii) The group $SL_{2}(\mathbb{Q}_{l})$ (where $\mathbb{Q}_{l}$ is the field of $l$-adic numbers) has property $\Tlp$ (see Exemple 9 in \cite{BO}). On the other hand, $SL_{2}(\mathbb{Q}_{l})$ is known to act on a tree without fixed point. Hence it does not have property $\Flp$ for any $p\geq1$  (see \cite{BHV} section 2.3 p.87 for instance). So property $(T_{\ell_{p}})$ is stricly weaker than property $(F_{\ell_{p}})$. Moreover, $G=SL_{2}(\mathbb{Q}_{l})$ is an example of a group such that $\mathcal{F}_{\ell_{\infty}}(G)$ is empty.\\
(iii) Any group $G$ with property $(T)$ has a spectrum $\mathcal{F}_{\ell_{\infty}}(G)$ of the form $[1,\pc[$ or $[1,p_{c}]$ for some $\pc>2$.
}
\end{examples}

As mentioned above, a second countable Kazhdan group has property $\Flp$ for all $1\leq p \leq 2$. This suggests the following question.
\begin{q}{\rm
Does there exist a second countable non-Kazhdan totally disconnected group which has property $(F_{\ell_{p}})$ for some (all) $1<p<2$ ?
}
\end{q}


\section{Fixed-point spectrum for actions on $L_{p}(X,\mu)$ associated to non-atomic measure spaces $(X,\mu)$}

This section is devoted to the study of connectedness properties for $\mathcal{F}_{L_{\infty}(X,\mu)}(G)$ for general measure spaces $(X,\mu)$. We give partial results concerning C. Drutu's conjecture asserting that $\mathcal{F}_{L_{\infty}(0,1)}(G)$ is connected.

\subsection{Openness of the spectrum $\mathcal{F}_{L_{\infty}}(G)$}

The openness of $\mathcal{F}_{L_{\infty}(0,1)}(G)$ at 2 is a well-known fact, due to Fisher and Margulis. Their argument uses a limit action on an ultraproduct of $L_{p}$-spaces, which we recall briefly in this section. Then we explain how it is used to show some openness properties relative to the fixed-point spectrum $\mathcal{F}_{L_{\infty}(0,1)}(G)$. We refer to the survey \cite{H} for more details on ultraproducts of Banach spaces.\\

We now recall the construction of the ultraproduct space of $L_{p}$-spaces in the context which is relevant for our purpose : we will use the ultraproduct of $L_{p_{n}}(X,\mu)$-spaces over the same measure space $(X,\mu)$, but such that $(p_{n})_{n}$ is a sequence of real numbers converging to $p\geq1$.\\
Let $1\leq p,p_{n}<\infty$ be real numbers such that $\lim_{n}p_{n}=p$. Let $(X,\mu)$ be a measure space. Fix a non-principal ultrafilter $\mathcal{U}$ on $\mathbb{N}$. We recall the construction of the ultraproduct (affine) space of the $L_{p_{n}}(X,\mu)$-spaces with marked points $x_{n}\in L_{p_{n}}(X,\mu)$. The latter affine space is defined as
$$(x_{n})_{n}+(L_{p_{n}})_{\mathcal{U}}=(x_{n})_{n}+(\prod_{n}L_{p_{n}})_{\infty}/\mathcal{N}$$ 
where 
$$(\prod_{n}L_{p_{n}})_{\infty}=\{\esp (f_{n})_{n}\esp\vert \esp \sup_{n}\n f_{n}\n_{p_{n}}<\infty \esp\},$$
and 
$$\mathcal{N}=\{\esp (f_{n})_{n}\in(\prod_{n}L_{p_{n}})_{\infty} \esp\vert\esp \n (f_{n})_{n}\n_{\mathcal{U}}=0\esp\}\textrm{ for }\n (f_{n})_{n}\n_{\mathcal{U}}:=\lim_{n,\mathcal{U}}\n f_{n}\n_{p_{n}}.$$
Recall the definition of a Banach lattice (from Definition 1.a.1 in \cite {LT}).
\begin{df}\label{df-Banach lattice}{\rm
A partially ordered Banach space $B$ over $\mathbb{R}$ is called a Banach lattice if the following assertions hold:
\begin {enumerate}
\item $x \leq y$ implies $x+z \leq y+z$ for every $x,y,z \in B$.
\item $\alpha x \geq 0$ for every $x \geq 0$ and real $\alpha \geq 0$.
\item for every $x,y \in B$ there exist a least upper bound $x \bigwedge y $ and a greatest lower bound $x \bigvee y$.
\item $\n x \n\leq \n y \n$ whenever $\vert x\vert \leq \vert y \vert$, where $\vert x \vert$ is defined as $\vert x \vert = x \bigwedge (-x)$. 
\end{enumerate}
}
\end{df}
$L_{p}(X,\mu)$-spaces are examples of Banach lattices. We recall also the definition of a $p$-additive norm.
\begin {df}\label{df-p additive Banach lattice}{\rm
Let $1 \leq p < \infty$. A norm on a Banach lattice is called $p$-additive if $\n x+y \n^{p} = \n x \n^{p} + \n y \n^{p}$, whenever $x\bigwedge y=0$.
}
\end{df}
An ultraproduct of Banach lattices is still a Banach lattice. Moreover, the norm $\n.\n$ on $(L_{p_{n}})_{\mathcal{U}}$ is clearly $p$-additive since $\lim_{n}p_{n}=p$. Hence by the generalized Kakutani representation theorem (see Theorem 1.b.2 in \cite{LT}), $(L_{p_{n}})_{\mathcal{U}}$ is isometrically isomorphic to $L_{p}(Y,\nu)$ for some measure space $(Y,\nu)$. \\

Let $G$ be a locally compact group generated by a compact subset $Q\subset G$. Let $\alpha_{n}=(\pi_{n},b_{n})$ be affine isometric actions of $G$ on the spaces $L_{p_{n}}(X,\mu)$. In the sequel, the diameter of a set $X$ is denoted by ${\rm diam}(X)$. Under the assumption that ${\rm diam}(\alpha_{n}(Q)x_{n})$ is bounded for all $n$, we can define an isometric affine action $\alpha$ on the affine space $(x_{n})_{n}+(L_{p_{n}})_{\mathcal{U}}$ by the following formula:
$$\alpha(g)((x_{n})+(f_{n})_{\mathcal{U}})=(x_{n})_{n} + (\alpha_{n}(g)x_{n}-x_{n}+\pi_{n}(g)f_{n})_{n}$$
for all $(f_{n})_{\mathcal{U}}\in (L_{p_{n}})_{\mathcal{U}}$ and all $g\in G$. This is not clear that the limit action needs to be continuous when $G$ is only assumed to be locally compact. We will discuss this issue later in this section.


\begin{thm}{\rm (Margulis, Fisher section 3.c in \cite{BFGM})}\label{thm-FisherMargulis}
Let $G$ be a finitely generated group. Assume $G$ has property $(T)$. Then the fixed-point spectrum $\mathcal{F}_{L_{\infty}(0,1)}(G)$ contains a neighborhood of $2$.
\end{thm}
We now sketch the proof of this theorem as it is given in section 3.c in \cite{BFGM}, since we will need some variation of it in the sequel.

\medskip

\begin{proof}
Let $Q\subset G$ be a compact generating set. The theorem is a consequence of the following claim.\\
Claim : there exists $C>0$, $\epsilon>0$ such that for all $q\in(2-\epsilon,2+\epsilon)$, for all affine isometric action $\alpha$ on $L_{q}(0,1)$, and for all $x\in L_{q}(0,1)$, there exists $y\in L_{q}(0,1)$ such that:\\
\begin{displaymath}
\begin{split}
& \n x-y\n_{q}\leq C{\rm diam}(\alpha(Q)x)\\
& {\rm diam}(\alpha(Q)y) \leq \frac{1}{2}{\rm diam}(\alpha(Q)x) .
\end{split}
\end{displaymath}
From the claim, it is very easy to show that $G$ has property $(F_{L_{q}(0,1)})$ (see \cite{BFGM}).\\

To prove the claim, we assume the contrary and show that this contradicts the fact that $G$ has property $(T)$. Hence we assume that there exists a sequence of reals $(p_{n})_{n}$ converging to 2, a sequence $(x_{n})_{n}$ with $x_{n}\in L_{p_{n}}(0,1)$, and affine isometric actions $\alpha_{n}=(\pi_{n},b_{n})$ on $L_{p_{n}}(0,1)$ such that ${\rm diam}(\alpha_{n}(Q)x_{n})=1$ and 
$${\rm diam}(\alpha_{n}(Q)y)\geq 1/2\textrm{ for all }y\in B(x_{n},n)\esp (\ast)\esp.$$
Then we define a limit action $\alpha$ on the ultraproduct affine space $(x_{n})_{n}+(L_{p_{n}}(0,1))_{\mathcal{U}}$, as recalled in the previous discussion. Since $\lim_{n}p_{n}=2$, the limit space is an affine Hilbert space and condition $(\ast)$ implies that $\alpha$ has no $G$-fixed-point, contradicting property $(T)$.
 \\

\end{proof}

\medskip


In the previous proof, one major argument requires the group $G$ to be countable: the continuity of the limit action. The following example shows that the continuity of the limit action does not hold in the general case of localy compact groups.

\begin{example}\label{ex-counterexample}{\rm
Let $G=\mathbb{R}$. Let $\mathcal{U}$ be a non-principal ultrafilter on $\mathbb{N}$. We define $\mathcal{H}=(L_{2}(\mathbb{R}))_{\mathcal{U}}$ to be the ultrapower of copies of $L_{2}(\mathbb{R})$ along $\mathcal{U}$. On the $n$-th copy of $L_{2}(\mathbb{R})$, define the orthogonal representation $\pi_{n}:G\rightarrow O(L_{2}(\mathbb{R}))$ by
$$\pi_{n}(a)f(x)=f(x+na)\textrm{ for all }n\in\mathbb{N}, a,x\in\mathbb{R}\textrm{ and }f\in L_{2}(\mathbb{R}).$$
Now let $\pi$ be the natural limit action (which acts by linear isometries) on $\mathcal{H}$ associated to the actions $\pi_{n}$. Denote also by $f$ the diagonal embedding of $f:=\chi_{[0,1]}$, the indicator function of $[0,1]$, in $\mathcal{H}$. It is easily checked that, for all $a\neq0$, we have
$$\n \pi(a)f-f\n=2.$$
Hence the limit action $\pi$ is not continuous on $\mathcal{H}$.
}
\end{example}

For countable groups $G$, one can show the openness of the fixed-point spectrum $\mathcal{F}_{L_{\infty}(0,1)}(G)$ in $[1,\infty[$.

\begin{prp}\label{prp-openness}
Let $G$ be a finitely generated group. Then $\mathcal{F}_{L_{\infty}(0,1)}(G)$ is open in $[1,\infty[$.
\end{prp} 

\begin{proof}
The proof of this result is also based on the construction of Margulis and Fisher. Let $p\geq1$ and assume that $G$ has property $(F_{L_{p}(0,1)})$. It is sufficient to prove an analog claim as in the proof of Theorem \ref{thm-FisherMargulis}, replacing $2$ by $p$, and to conclude with the arguments from \cite{BFGM}. Let $Q\subset G$ be a finite generating set.\\
Claim : there exists $C>0$, $\epsilon>0$ such that for all $q\in(p-\epsilon,p+\epsilon)$, for all affine isometric action $\alpha$ on $L_{q}(0,1)$, and for all $x\in L_{q}(0,1)$, there exists $y\in L_{q}(0,1)$ such that:\\
\begin{displaymath}
\begin{split}
& \n x-y\n_{q}\leq C{\rm diam}(\alpha(Q)x)\\
& {\rm diam}(\alpha(Q)y) \leq \frac{1}{2}{\rm diam}(\alpha(Q)x) .
\end{split}
\end{displaymath}
If the claim does not hold, we can construct isometric affine actions $\alpha_{n}$ of $G$ on $L_{p_{n}}(0,1)$ for all $n$, and a limit isometric affine action $\alpha=(\pi,b)$ of $G$ on a space $(x_{n})_{n} + (L_{p_{n}}(0,1))_{\mathcal{U}}$ such that ${\rm diam}(\alpha_{n}(Q)x_{n})=1$, $\lim_{n}p_{n}=p$, and
$${\rm diam}(\alpha_{n}(Q)y)\geq 1/2\textrm{ for all }y\in B(x_{n},n)\esp (\ast)\esp.$$
From inequality $(\ast)$, the action $\alpha$ has no fixed point, and so the cocycle $b$ is unbounded.\\
By the Kakutani theorem, the space $(L_{p_{n}}(0,1))_{\mathcal{U}}$ is isometrically isomorphic to $L_{p}(Y,\nu)$ for some measure space $(Y,\nu)$. Hence the orthogonal representation $\pi$ can be assumed to act on $L_{p}(Y,\nu)$, and the cocycle part $b$ to satisfy $b(g)\in L_{p}(Y,\nu)$ for all $g\in G$.\\
By Lemma 9.2 in \cite{NP}, there exists an isometric affine action $\alpha'=(\pi',b')$ on $L_{p}(0,1)$ such that $\n b(g)\n_{p}=\n b'(g)\n_{p}$ for all $g\in G$. Since $b$ is unbounded, the cocycle part $b'$ is unbounded as well, and $\alpha'$ has no fixed point in $L_{p}(0,1)$. Notice that the action $\alpha'$ is continuous since $G$ is discrete. So $G$ does not have property $(F_{L_{p}(0,1)})$, which is a contradiction. Hence the claim holds and Proposition \ref{prp-openness} is proved. $\bsq$
\end{proof}

In the previous proof, two major arguments require the group $G$ to be discrete: the continuity of the limit action (which is necessary as discussed previously), and the restriction of an arbitrary $L_{p}(Y,\nu)$ limit space on which $G$ acts on, to the classical $L_{p}(0,1)$ space. Unfortunately, we are not able to extend Proposition \ref{prp-openness} to second countable groups since we do not have a proof for the continuity of the limit action in that case. An analog of the second argument should be easier to extend from the same lines of proof of Lemma 9.2 in \cite{NP}. So the continuity issue seems to be the main difficulty to generalize Proposition \ref{prp-openness} to second countable groups.


\subsection{Fixed-point spectrum associated to measure preserving ergodic actions on finite measure spaces}

As a consequence of the following proposition, the fixed-point spectrum $\mathcal{F}_{L_{\infty}(X,\mu)}(G,\pi)$ is an interval when the measure $\mu$ is finite, and the action associated to the representation $\pi$ is measure-preserving and ergodic.

\begin{prp}\label{prp-mpinterval}
Let $1\leq p<q<r<\infty$. Let $(X,\mu)$ be a Borel standard measure space such that $\mu$ is finite. Let $\pi^{p}:G\rightarrow O(L_{p}(X,\mu))$ be a (BL)  representation whose associated action is measure-preserving and ergodic. Assume that 
$$H^{1}(G,\pi^{p})=\{0\},$$
and
$$H^{1}(G,\pi^{r})=\{0\}.$$
Then we have also
$$H^{1}(G,\pi^{q})=\{0\}.$$
\end{prp}
The proof follows the same lines as the proof for actions on $\ell_{p}$.\\

\begin{proof}
Notice that we have $L_{r}(X,\mu)\subset L_{q}(X,\mu)\subset L_{p}(X,\mu)$.\\
Let $b:G\rightarrow L_{q}(X,\mu)$ be a $\pi^{q}$-cocycle. Since $\pi=\pi^{p}$ is measure-preserving, $b$ defines a cocycle with values in $L_{p}(X,\mu)$. By assumption, there exists $f\in L_{p}(X,\mu)$ such that
$$b(g)=f-\pi^{p}(g)f \textrm{ for all }g\in G.$$
Triangle inequalities imply the following inequalities for all $g\in G$:
$$\n \vert f\vert - \vert\pi\vert(g) \vert f\vert\n_{q} \leq \n b(g)\n_{q}<\infty.$$
Now we apply estimates (2) from section 2.2 with $q<r$, and $\vert f\vert$ to obtain:
$$\n \vert f\vert^{q/r} - \vert\pi\vert(g)\vert f \vert^{q/r}\n_{r}  \leq \n \vert f\vert - \vert\pi\vert(g) \vert f\vert\n_{q}^{q/r}<\infty.$$
Hence the left hand-side of the previous inequality defines a cocycle with values in $L_{r}$, which is a coboundary by assumption. Since the action of $G$ on $(X,\mu)$ is ergodic we can write:
$$\vert f\vert^{q/r}=c+h$$
where $c$ is a constant function, and $h\in L_{r}(X,\mu)'$. Then we have $\vert f\vert^{q/r}\in L_{r}(X,\mu)$, that is $f\in L_{q}(X,\mu)$. $\bsq$
\end{proof}

\medskip

Now the proof of Theorem \ref{thm-main-mp} is an easy consequence of Proposition \ref{prp-mpinterval}.

\begin{proof-thm3}
Let $\pi:G\rightarrow O(L_{p}(X,\mu))$ be a measure-preserving ergodic on $(X,\mu)$ finite. By Proposition \ref{prp-mpinterval}, the set $\mathcal{F}_{L_{\infty}(X,\mu)}(G,\pi)$ is empty or is an interval. When the fixed-point spectrum is non-empty, it is closed on the right by proposition \ref{prp-closedness}.\\
If $G$ has property $(T)$, we know that the set $\mathcal{F}(G)$ contains an interval of the form $[1,q[$ for some $q>2$ (see Theorem \ref{thm-FisherMargulis} from section 5.1, and Theorem 1.3 in \cite{BFGM}). Combined with our Proposition \ref{prp-mpinterval}, the set $\mathcal{F}_{L_{\infty}(X,\mu)}(G,\pi)$ is an interval of the form $[1,p_{c}]$ or $[1,\infty[$ for some $p_{c}>2$. Hence the theorem is proved. $\bsq$ 
\end{proof-thm3}

\medskip

Let $p\geq1$. By a construction from \cite{CW}, in the case where $G$ does not have property $(T)$, there exists a (Gaussian) finite measure space $(X,\mu)$ endowed with an action of $G$ such that :\\
- the action of $G$ on $(X,\mu)$ is measure-preserving;\\
- the restriction of the associated representation $\rho^{p}$ to $L_{p}'(\rho^{p})$ admits a sequence of almost invariant vectors.\\
Thus by the usual Guichardet argument, we have
$$\H1(G,\rho^{p})\neq\{0\},$$
and it follows that $\mathcal{F}_{L_{\infty}(X,\mu)}(G,(\rho^{p})_{p})=\emptyset$ when $G$ does not have property $(T)$.\\
The construction of such an action of $G$ on some Gaussian measure space $(X,\mu)$ is due to Connes and Weiss, and details are explained in \cite{BHV} (Theorem 6.3.4). See also section 4.c in \cite{BFGM}, for the proof of the statement related to almost invariant vectors.\\

Let $(X,\mu)$ and $\pi$ be defined as in Proposition \ref{prp-mpinterval}. In the case where $G$ has property $(T)$, we know by Proposition \ref{prp-openness} that $\mathcal{F}_{L_{\infty}(0,1)}(G)$ is open in $[1,\infty[$. If a similar openness result was true for fixed-point spectrum $\mathcal{F}_{L_{\infty}(X,\mu)}(G,\pi)$ restricted to measure-preserving and ergodic actions, we could conclude that $\mathcal{F}_{L_{\infty}(X,\mu)}(G,\pi)$ is equal to the whole interval $[1,\infty[$. Unfortunately, we could not derive an analog of Proposition \ref{prp-openness} for fixed-point spectrum sets of the form $\mathcal{F}_{L_{\infty}(X,\mu)}(G,\pi)$.

\begin{q} {\rm
If $\pi$ is a measure-preserving ergodic (BL) representation on a finite measure space $(X,\mu)$, and if $G$ has property $(T)$, do we always have $\mathcal{F}_{L_{\infty}(X,\mu)}(G,\pi)=[1,\infty[$ ?
}
\end{q}

\noindent
{\bf Address}
\medskip

\noindent
Omer Lavy.

\noindent
Weizmann Institute of Science, Department of Mathematics \\
Herzl St 234, Rehovot, Israel\\
omer.lavi@weizmann.ac.il

\bigskip

\noindent
Baptiste Olivier.

\noindent
Orange Labs\\
4 rue clos courtel, 35510 Rennes, France\\
baptiste.olivier100@gmail.com

\end{document}